\documentclass[11pt,reqno]{amsart}

\usepackage[page]{appendix}
\usepackage{a4wide}
\usepackage{amsmath}
\usepackage{amsfonts}
\usepackage{amssymb}
\usepackage{graphicx}
\usepackage{cite}
\usepackage{color}
\usepackage{subfigure}
\usepackage{hyperref}

\newtheorem{Lemma}{Lemma}
\newtheorem{Proposition}{Proposition}
\newtheorem{remark}{Remark}
\newtheorem{Theorem}{Theorem}
\newtheorem{ATheorem}{Theorem}[section]
\newtheorem{ALemma}{Lemma}[section]
\newtheorem{Definition}{Definition}
\newtheorem*{Definition*}{Definition}
\newtheorem{Corollary}{Corollary}

\newcommand{\deq}{\stackrel{d}{=}}
\newcommand{\mean}{\mathbb E}
\newcommand{\prob}{\mathbb P}

\begin{document}
\pagestyle{empty}
\centerline{\Large\bf Queues and risk processes with dependencies}
\bigskip
\centerline{E.S.\ Badila\footnotemark[1], O.J.\ Boxma \and J.A.C.\ Resing}
\address{Department of Mathematics and Computer Science, Eindhoven University of Technology, P.O. Box 513, 5600 MB Eindhoven, The Netherlands.}
\footnotetext[1]{ Supported by Project 613.001.017 of the Netherlands Organisation for Scientific Research (NWO)\\ \indent\emph{Email addresses: e.s.badila@tue.nl, o.j.boxma@tue.nl, resing@win.tue.nl.}}
\vspace{2cm}
\setlength{\parindent}{0pt}
\textbf{Abstract:}
We study the generalization of the $G/G/1$ queue obtained by relaxing the assumption of independence between inter-arrival times and service requirements. The analysis is carried out for the class of multivariate matrix exponential distributions introduced in \cite{Bladt:MVME}.
In this setting, we obtain the steady state waiting time distribution and we show that the classical relation between the steady state waiting time and the workload distributions remains valid when the independence assumption is relaxed. We also prove duality results with the ruin functions in an ordinary and a delayed ruin process. These extend several known dualities between queueing and risk models in the independent case.
 Finally we show that there exist stochastic order relations between the waiting times under various instances of correlation.

\vspace{1cm}

\vspace{0.2cm}

\begin{section}{Introduction}
In this paper we study a single server queue with the special feature
that the service requirement of each arriving customer is correlated with the subsequent inter-arrival time.
Dependence between service and inter-arrival times arises naturally in a number of applications.
If one has some control over the arrival process to the server, then one might, e.g.,
wait a relatively long (short) time with dispatching a new job to the server, if the previous job was relatively big (small).
In fact, we shall see in Section \ref{Section numerics} that a positive correlation
between the service requirement and the subsequent inter-arrival time
reduces the waiting times, whereas negative correlation increases waiting times. The increase/decrease is in the sense of \emph{convex ordering} (cf. \cite{Stoyan&Daley}, Ch. 1).

In studying the single server queue $G/G/1$, it is usually assumed that all inter-arrival times and service requirements are independent.
An important exception is the class of queues with Batch Markovian Arrival Process, $BMAP/G/1$, see for example Lucantoni
\cite{Lucantoni} and references therein.
The $BMAP/G/1$ queue provides a framework to model dependence between successive interarrival times. In \cite{CombeBoxma} it is also used to study an $M/G/1$ queue
in which service requirements depend on the previous inter-arrival times; see \cite{BBC} for a different approach
to the latter form of dependence, which does not use the MAP machinery.
An important paper regarding dependence between inter-arrival and service requirements is the one by Adan and Kulkarni \cite{AK}.
They consider a single server queue with Markov-dependent inter-arrival and service requirements:
a service requirement and subsequent inter-arrival time have a bivariate distribution that depends on an underlying Markov chain
which jumps at customer arrival epochs. The inter-arrival times in \cite{AK} are exponentially distributed, with rate $\lambda_j$ when the Markov chain jumps to state $j$.

It should be observed that the analysis of a $G/G/1$ queue with some dependence structure between a service requirement $B_i$ and the {\em subsequent} inter-arrival time $A_i$ is intrinsically easier than that of a
$G/G/1$ queue with some dependence structure between $A_i$ and the {\em next} $B_{i+1}$.
The reason is that $B_i$ and $A_i$ only appear as a difference $B_i-A_i$ in the Lindley recursion $W_{i+1} = {\rm max}(W_i+B_i-A_i,0)$
for the waiting time $W_i$ of the $i^{th}$ arriving customer.
In a sense, the study of the waiting time distribution in the $G/G/1$ queue reduces to the study of a random walk
with steps $B_i-A_i$.
Still, there are not many examples known of joint distributions of $(B_i,A_i)$ that
allow a detailed exact analysis. One of the exceptions is provided in \cite{BP}, where a threshold-type dependence
between $B_i$ and $A_i$ is
shown to be analytically tractable.

In the present study, we shall consider a very general class
of bivariate distributions of $(B_i,A_i)$, which allows us to obtain
detailed, explicit, results for the steady-state waiting time and workload distribution.
The dependence structure under consideration is modelled by a class of bivariate matrix-exponential distributions (Bladt and Nielsen~\cite{Bladt:MVME}) in which the joint Laplace-Stieltjes transform of the claim size and the inter-claim time is a rational function. 

While this paper was under preparation, Hansjoerg Albrecher kindly pointed out to us that
Constantinescu et al. \cite{Cstescu}  were obtaining results similar to ours for a generalization of
the Sparre Andersen insurance risk model.
The classical Sparre Andersen model considers the development of the capital
of an insurance company that earns premium at a fixed rate and that receives claims with a stochastic size at stochastic inter-arrival times --
all the input variables being {\em independent}.
In contrast, Constantinescu et al. \cite{Cstescu} allow a claim size to depend on the previous inter-claim time,
in a similar way as an inter-arrival time depends on the previous service requirement in our queueing model.
One can establish a duality relation between the insurance risk model of \cite{Cstescu} and our model
(cf.\ Section~\ref{S4}), and this duality relation in particular
implies that the probability of ruin of the insurance company, with initial capital $u$, equals the probability
that the steady-state waiting time in the corresponding queueing model exceeds $u$.
Our approach is based on Wiener-Hopf factorization; Constantinescu et al. \cite{Cstescu} use a completely different approach, based on operator theory methods.
We shall explore the relation between the queueing and insurance risk models with dependence in more detail,
which will allow us to also obtain the so-called delayed ruin probability
in the model of \cite{Cstescu}, viz., the ruin probability when time $0$ is not a claim arrival epoch but an arbitrary epoch,
the claim arrival process being in stationarity.

Already having discussed the queueing literature with dependence between inter-arrival and service requirement,
let us now turn to the insurance risk literature with dependence between inter-claim time and claim size.
In recent years, this has been a hot topic in risk theory.
Albrecher and Boxma~\cite{Albrecher&Boxma2004} derive exact formulas for the ruin probability in a Cram\'er-Lundberg model with a threshold-type dependence between a claim size and the next inter-claim time.
In \cite{Albrecher&Boxma2005} a much more general semi-Markovian risk model is being considered,
which bears some resemblance to the queueing model in \cite{AK}.
Kwan and Yang \cite{KY} consider a specific threshold-type dependence of claim size on previous inter-claim time;
in \cite{ABI} this is put in the larger framework of Markov Additive Processes.
Another specific dependence structure between claim size and previous inter-claim time
is treated in Boudreault et al.\ \cite{Boudreault}.
Asymptotic results were obtained in Albrecher and Kantor~\cite{Albrecher&Kantor2002}, where the relation between the dependence structure and the Lundberg exponent is studied. Also Albrecher and Teugels~\cite{Albrecher&Teugels2006} give asymptotic results for the finite and infinite horizon ruin probabilities when the current claim size and the previous inter-claim time are dependent according to an arbitrary copula structure.


The {\em main contributions} of the paper are the following.
(i) We provide an exact analysis of the waiting time distribution in a $G/G/1$ queue
with correlation between a service requirement $B$ and the subsequent interarrival time $A$,
$B$ and $A$ having a multivariate matrix-exponential distribution.
(ii) We prove that the simple relation which holds between steady-state workload and waiting time distributions in the ordinary
$G/G/1$ queue remains valid in the case of correlated $B$ and $A$.
(iii) We consider the dual Sparre Andersen insurance risk model with correlation between inter-claim time and subsequent claim size, and in particular we show that the Tak\'acs relation (cf.~ \cite{Franken}, Corollary 4.5.4) between the ordinary ruin probability
and the delayed ruin probability remains valid. (iv) Finally, we show that, in comparison with the classical set-up without dependence, positive and negative correlation respectively decreases and increases the waiting times in the sense of convex ordering. We also illustrate with numerical results the influence of dependence on the expected values of the waiting times but also on the $95\%$-percentiles of the ruin functions (VaR's).

The paper is organized as follows.
Section~\ref{S2} contains a detailed model description, which in particular includes a description of the class of bivariate distributions under consideration. It also presents the waiting time analysis.
The relation between the steady-state waiting time and workload distributions is exposed in Section~\ref{S3}.
Section~\ref{S4} is devoted to the dual insurance risk model.
In Section~\ref{Section numerics} we consider several examples of bivariate distributions of $B_i$ and $A_i$.
For these examples, we present numerical results on the mean and tail of the waiting time distribution
(and, by duality, on the ruin probability), which exhibit the effect of (positive or negative) correlation
on waiting time and ruin probability, together with stochastic ordering results and by consequence, ordering between the waiting times.
\end{section}

\begin{section}{Model Description and Analysis of the waiting time \label{S2}}
We study a generalization of the classical $G/G/1$ model, where we allow for an arbitrary correlation between the service requirement of the $n^{th}$ customer and the inter-arrival time between the $n^{th}$ and $(n+1)^{th}$ customer.
As a key performance measure in this model, we first consider the waiting time process in an initially empty system. In Section \ref{S3}, we prove that the steady-state waiting time is related to the steady-state workload in a similar way as in the independent case.

Let $B_i$ be the service requirement of the $i^{th}$ customer, $A_i$ the inter-arrival time between the $i^{th}$ and the $(i+1)^{th}$ customer, and $c$ the server's speed.
We assume that $(A_{i},B_i)$ are i.i.d. sequences of random vectors. This implies that the arrival process of customers is renewal and that the quantities $(B_i-cA_{i})$ are i.i.d. However, within a pair, $A_i$ and $B_i$ are dependent, hence the $i^{th}$ service requirement and the subsequent inter-arrival time are correlated. We denote by $(B,A)$ a generic pair made up of a service requirement and the subsequent inter-arrival time.
In Figure~\ref{fig1} we display the workload process $\{V_t, t \geq 0\}$
and the waiting time process $\{W_n, n=1,2,\dots\}$; here $V_t$ denotes work in the system at time $t$, and $W_n$ denotes the waiting time of the
$n^{th}$ arriving customer.
The waiting time process satisfies the Lindley recursion:
\begin{equation*} W_{n+1}= \max(W_{n}+c^{-1}B_n-A_n,0).\end{equation*}

\vspace{0.2cm}
Under the stability condition $\mathbb E(c^{-1}B-A)<0$, $W_n$ converges in distribution to a proper random variable $W$ and we can write:
\begin{equation}W\deq\max\left(W+c^{-1}B-A,0\right). \label{stationary Lindley}\end{equation}

\begin{center}
\begin{figure}[!ht]
\includegraphics[width=1\linewidth,clip=true,trim=2cm 18cm 0cm 4cm]{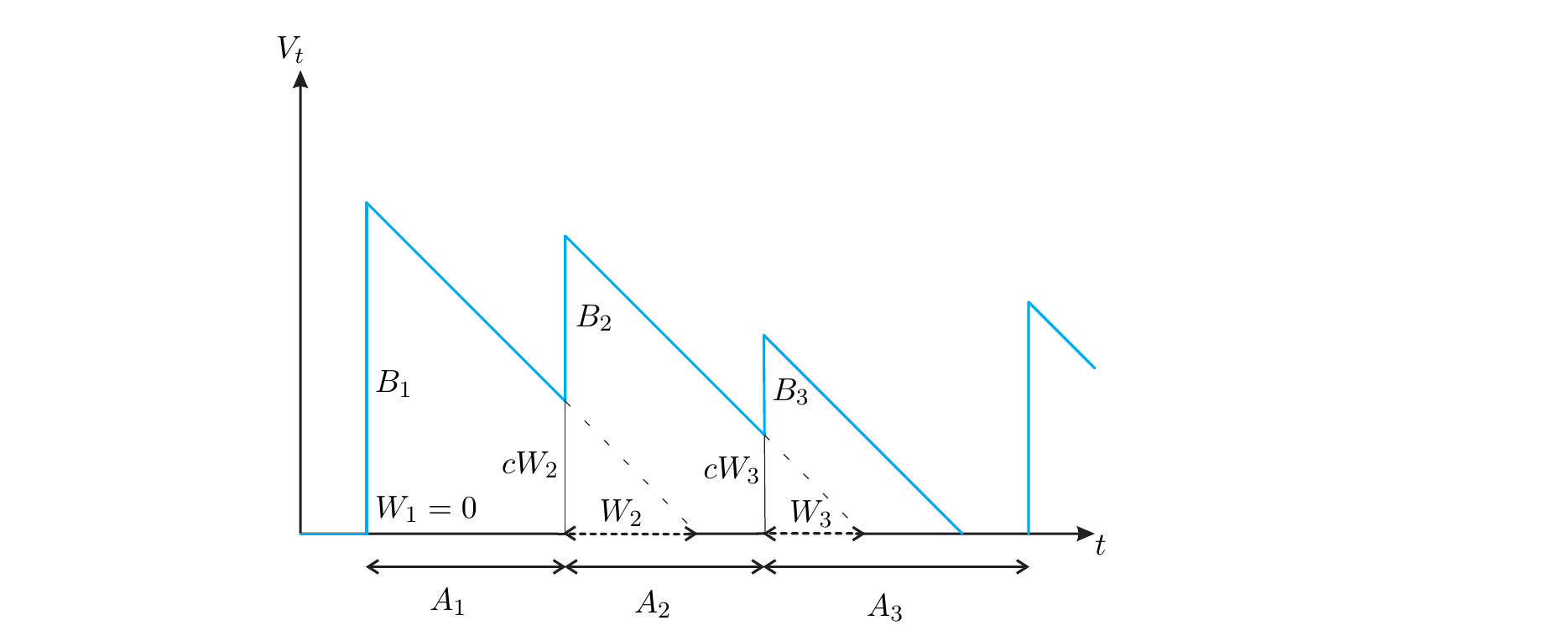}
\caption{Workload process and waiting time process}
\label{fig1}
\end{figure}
\end{center}

\paragraph{\textbf{The dependence structure:}}

We model the dependence structure using the class of multivariate matrix-exponential distributions (MVME), which was introduced by Bladt and Nielsen~\cite{Bladt:MVME}. This class contains other known classes of distributions with interesting probabilistic interpretations, like the multivariate phase-type distributions studied in  Assaf et al.~\cite{Assaf:MPH} and further in Kulkarni~\cite{Kulkarni:MPH*}.
We will further discuss this class in Section \ref{Section numerics} where we also give examples which admit a probabilistic interpretation.
Below we cite Definition 4.1 of Bladt and Nielsen~\cite{Bladt:MVME}:

\begin{Definition}A non-negative random vector $(A,B)$ is said to have a  bivariate matrix-exponential distribution if the joint Laplace-Stieltjes transform (LST) $\mathbb Ee^{-s_1A-s_2B} $ is a rational function in $(s_1,s_2)$, i.e. it can be written as $\frac{F(s_1,s_2)}{G(s_1,s_2)}$, where $F(s_1,s_2)$ and $G(s_1,s_2)$ are polynomial functions in $s_1$ and $s_2$.
\end{Definition}

\noindent As a consequence of this defining property, the transform of the difference $Y:=c^{-1}B-A$ is also a rational function. 
For simplicity, let us denote $\mathbb E {\rm e}^{-sY} :=\frac{f(s)}{g(s)}$.
 We rewrite identity (\ref{stationary Lindley}) in terms of Laplace-Stieltjes transforms. After some straightforward computations, one obtains:

\begin{equation} \mathbb E e^{-sW}\left[1-\mean e^{-sY}\right]= \mathbb P\left( W+Y\leq 0\right)  -  \mathbb Ee^{-s(W+Y)}1_{\left\{W+Y\leq 0\right\}}.\label{analyt} \end{equation}

\noindent Using the rationality of the transform of $Y$, we can rewrite (\ref{analyt}):
\[\mathbb E e^{-sW}\frac{g(s)-f(s)}{g(s)}= R_-(s),\]
where $R_-(s)$ is the function on the right-hand side of (\ref{analyt}), which is analytic in $\mathcal Re\,s<0$ and continuous in $\mathcal Re\,s\leq 0$. Also, since $W\geq 0 $ by definition, $\mathbb E e^{-sW}$ is analytic in $\mathcal Re\, s> 0$ and continuous in $\mathcal Re\, s \geq 0$.
\\
Using Wiener-Hopf factorization, we now obtain the LST of the steady-state waiting time distribution:

\begin{Theorem}\label{theorem delay}For $(A,B)$ having a bivariate matrix exponential distribution, the LST of the steady state waiting time is given by
 \begin{equation}\label{delay transform}E e^{-sW}= \frac{\prod_{\tilde{s}^-_j}(1-\frac{s}{\tilde{s}^-_j})}{\prod_{s^-_k}(1-\frac{s}{s^-_k})},\end{equation} where $s^-_k$ are the zeros of $1-\mean e^{-sY}$ in $\mathcal Re\, s<0$ and $\tilde{s}^-_j$ are its poles in $\mathcal Re\;s<0$.
\end{Theorem}

\begin{proof} Let $m_+$ be the number of zeros of $g(s)$ in $\mathcal Re\; s \geq 0$. We move these to the right-hand side of the identity above:
\begin{equation}
\mathbb Ee^{-sW}\frac{g(s)-f(s)}{g_-(s)}=g_+(s)R_{-}(s),
\label{ident2}
\end{equation}
where $g_+(s)=\prod_{k=1}^{m_+}(s-\tilde s^+_k)$, the product being over the zeros of $g$ with $\mathcal Re\;\tilde s^+_k\geq 0$; and $g_-(s)=g(s)/g_+(s)$. Now the left-hand side of (\ref{ident2}) is analytic in $\mathcal Re\,s\geq 0$, the right-hand side remains analytic in $\mathcal Re\,s < 0$; therefore by analytic continuation, the left-hand side is an entire function.

We use a version of Liouville's theorem \ref{A2} (see Appendix), which states that an entire function with asymptotic behavior $O(|s|^{m_+})$ must be a polynomial of degree at most $m_+$.

Liouville's theorem  implies that the left-hand side of (\ref{ident2}) is a polynomial $P(s)$ of degree $deg(P)\leq deg(g_+)=m_+$.
 Therefore we can write \begin{equation}\mathbb Ee^{-sW}=\frac{g_-(s)}{g(s)-f(s)}P(s).\label{af}\end{equation} Since $g_-(s)$ has zeros only in $\mathcal Re\;s<0 $, $P(s)$ must have all the zeros of $g-f$ from $\mathcal Re\; s\geq 0$ because otherwise $\mathbb Ee^{-sW}$ would have a pole in $\mathcal Re\, s\geq 0$ which is not possible. 

\vspace{0.2cm}

Now all boils down to showing that $g(s)-f(s)$ and $g(s)$ have the same number of zeros (i.e. $m_+$) in $\mathcal Re\;s\geq 0$. Rouch\'e's theorem \ref{A1} in the Appendix seems to be the right tool for this, and in Lemma \ref{A3} in the Appendix we show that indeed $|g(s)|>|f(s)|$ in $\mathcal Re\;s\geq 0$.

Since $P(s)$ must have these $m_+$ zeros of $g(s)-f(s)$ as its own, and at the same time $deg(P)\leq m_+$ from above, this determines $P(s)$ up to a constant: $P(s) = C(g-f)_+(s)$, where  $(g-f)_+(s):=\prod_{s^+_k}(s-s^+_k)$, $s^+_k$ being the zeros of $(g-f)(s)$ with $\mathcal Re\; s\geq 0$ (this also includes the zero at $s_0=0$).  After replacing $P(s)$ and reducing the factors in Formula (\ref{af}), we obtain the following formula for $\mathbb E e^{-sW}$:
\begin{equation}\label{final W} \mathbb Ee^{-sW}=C\frac{\prod_{\tilde s^-_j} (s- \tilde s^-_j) }{\prod_{s^-_k} (s- s^-_k) }.  \end{equation}
Setting $s=0$ determines the constant: $C= \prod_{s^-_k} (- s^-_k) /\prod_{\tilde s^-_j} (- \tilde s^-_j) $, hence
(\ref{delay transform}) follows.
%
\end{proof}

\begin{remark} The PASTA property does not hold,
and hence the distribution of the steady-state workload differs in principle from $cW$, the steady-state workload as seen by an arriving customer.
In particular,
we have $\prob(V=0)\neq \prob(cW=0).$ Actually, we find the atom at zero of $cW$ if we take $s\rightarrow \infty$ in (\ref{delay transform}), with the additional remark that the numerator has the same number of factors as the denominator, which follows from Rouch\'e's theorem:
\begin{equation}\label{W atom zero}\prob(cW=0) = C = \prod_{s_k^{-}} s_k^{-}/ \prod_{\tilde s_j^{-}} \tilde s_j^{-}.\end{equation}
On the other hand, from first principles we have, with $\rho := \frac{\mean B}{c \mean A}$, for the steady-state probability of an empty system:
\[
\prob(V=0) = 1 - \rho .
\]
\end{remark}

The factorization used in the proof of identity (\ref{delay transform}) can be also used to obtain the transform $\mean e^{sI}$ of $I$, the steady state idle period of the system.

\begin{Corollary}The transform of the idle period is given by

\[\mean e^{sI} = 1 - \prod\limits_{s_k^+} (s-s_k^+)/\prod\limits_{\tilde s_j^+}(s-\tilde s_j^+), \,\,\,\mathcal Re\,s\leq 0,\]

with $s_k^+$ being the zeroes of $\frac{g(s)-f(s)}{g(s)}$ in $\mathcal Re\, s\geq 0$ and $\tilde s_j^+$ its poles in $\mathcal Re\, s>0.$
\end{Corollary}
\begin{proof}
Conditional on $W+Y\leq 0$, $I=-(W+Y)$, so we may write 

\[ \mean e^{sI} = \frac{1}{\prob(W+Y\leq 0)}\mean e^{s(-W-Y)}1_{\{W+Y\leq 0\}}.  \]
The transform $\mean e^{s(-W-Y)}1_{\{W+Y\leq 0\}}$ already appears on the right-hand side of (\ref{analyt}), hence the transform of the idle period can be rewritten as

\begin{equation}\label{idle period}\mean e^{sI} =1- \frac{1}{\prob(W+Y\leq 0)}\cdot  \mean e^{-sW}\cdot\frac{g(s)-f(s)}{g(s)}, \end{equation}

As in the proof of Theorem \ref{theorem delay}, we make use of the factorizations $g(s)=g_+(s)\cdot g_-(s)$ and $(g-f)(s)=(g-f)_+(s)\cdot(g-f)_-(s) $ which were obtained via Rouch\'e's theorem. Therefore, using (\ref{af}), (\ref{final W}) and (\ref{W atom zero}) we may write

\[ \mean e^{sI} = 1 - \frac{\prob(W=0)}{\prob(W+Y\leq 0)}\cdot \frac{g_-(s) }{(g-f)_-(s)} \cdot\frac{g(s)-f(s)}{g(s)}.  \]
Note that the identity in law (\ref{stationary Lindley}) implies $\prob(W=0) = \prob (W+Y\leq 0)$. After cancelling the factors above, $\mean e^{sI}$ reduces to
\[\mean e^{sI} = 1 - \prod\limits_{s_k^+} (s-s_k^+)/\prod\limits_{\tilde s_j^+}(s-\tilde s_j^+).\]
\end{proof}

\begin{remark}
Alternatively we can use Formula (6.20) in Cohen~\cite{Cohen1976}, p.21, which makes use of the regenerative structure of the workload process w.r.t. the busy cycles of the queue. It can be shown that the formula remains valid even in the dependent case. The connection with (\ref{idle period}) is then $\frac{1}{\prob(W+Y \leq 0)} = \mean N$, the mean number of customers served during a busy cycle.
\end{remark}
In the next section we show that similar arguments involving regeneration as the ones employed in \cite{Cohen1976}, can be extended in our setting to give the relation between the steady-state workload and waiting time distributions.
\end{section}

\begin{section}{The steady-state workload}
\label{S3}
In this section we consider the steady-state workload in the queueing model
with correlation between service requirement $B$ and subsequent inter-arrival time $A$.
We shall prove that the known relation between the steady-state workload and waiting time
for the single server queue with {\em independent} service requirement and inter-arrival time (\cite{Asmussen_queues}, p. 274,
\cite{Cohen1976}, p. 19/20, or
\cite{Cohen SingleServer}, p. 296/297)
remains valid.
For this purpose we adapt the proof in \cite{Cohen1976}, which is based on the fact that the workload process
regenerates at the beginning of each busy cycle. The LST of the workload and waiting time distributions
can then be written as stochastic mean values of the LST over one full busy cycle.
\begin{Theorem}The steady-state workload $V$ and the waiting time $W$ are related in the following way:

\begin{equation}\label{queueing relation formula}\prob(V\leq v)=1-\rho + \rho \,\prob(cW+ B^{res}\leq v), \end{equation} with $\rho =\frac{\mean B}{c\mean A}$ and $B^{res}$ the marginal distribution of a residual service requirement, viz.,
\[
\begin{array}{l}
\prob(B^{res}\leq v)=\frac{1}{\mean B}\int\limits_0^v \prob(B> u)\,{\rm d}u.
\end{array}
\]
\label{queueing relation}
\end{Theorem}
Remark that only the marginal distribution of the residual service requirement appears in the above, not the joint distribution of $A$ and $B$.
\begin{proof} Let 0 be the beginning of a busy period and $P$ be its length.
Following Cohen\cite{Cohen1976}, within this busy period, we may write (cf.\ Figure~\ref{fig1}): \[V_t=cW_{n(t)}+B_{n(t)} - c(t-t_{n(t)}),\] where
$V_t$ is the workload at time $t$,
$n(t)$ is the number of arrivals in $[0,t]$ and
$t_{n(t)}$ is the last arrival epoch before $t$. The following identities hold path-wise:

\[\int_0^{P}e^{-sV_t}{\rm d}t=\int_0^{P} e^{-s\left[cW_{n(t)}+B_{n(t)}-c(t-t_{n(t)})\right]}\,{\rm d}t\]
\begin{equation}=
\sum_{i=1}^{N-1} \int_0^{A_{i}}e^{-s(cW_i+B_i-ct)}{\rm d}t + \int_0^{A_{N}-I} e^{-s(cW_{N} +B_{N} -ct)}\,{\rm d}t. \label{regen1}
\end{equation}
Here $N$ is the number of customers served during a busy period. 
The key observation is that the following relation holds even when $A_i$ and $B_i$ are dependent:
\[\int_0^{A_{i}} e^{-s(cW_i+B_i-ct)}{\rm d}t = e^{-s(cW_i+B_i)}\frac{1}{cs}(e^{csA_{i}}-1).   \] There is no expectation taken so integration is carried out as usual, all these being path-wise identities. Formula  (\ref{regen1}) now becomes

\[\int\limits_{0}^Pe^{-sV_t}\,{\rm d}t= \frac{1}{cs}\sum\limits_{i=1}^{N-1} e^{-s(cW_i+B_i)}\left[ e^{csA_{i}}-1\right] + \frac{1}{cs}e^{-s(cW_{N}+B_{N})}\left[e^{cs(A_{N}-I)}-1\right] \]
\[= \frac{1}{cs}\sum\limits_{i=1}^{N-1}\left[e^{-s(cW_i+B_i-cA_i)}-e^{-s(cW_i+B_i)} \right] + \frac{1}{cs}e^{-s[cW_{N}+B_{N}-c(A_{N}-I)]}-\frac{1}{cs}e^{-s(cW_{N}+B_{N})}.\]

\noindent We make use of the following identities for the waiting time during a busy period:\newline For $i\leq N -1$, $cW_i+B_i-cA_i=cW_{i+1}$; and $cW_{N}+B_{N}-cA_{N}=-cI$, hence

\[\int\limits_{0}^Pe^{-sV_t}\,{\rm d}t= \frac{1}{cs}\sum\limits_{i=1}^{N-1}\left( e^{-scW_{i+1}}-e^{-scW_i-sB_i} \right) + \frac{1}{cs}\left[1-e^{-scW_{N}-sB_{N}} \right]\]

\begin{equation}\label{regen total workload}=\frac{1}{cs}\sum\limits_{i=1}^{N}e^{-scW_i}(1-e^{-sB_i}).\end{equation}
All derivations up to this point are path-wise manipulations, hence insensitive to correlations between $A_{i}$ and $B_i$.
Remark that $B_n$ is independent of $W_n$ but also of the r.v. $1_{\{N\geq n \}}$.
So if we take expectations in (\ref{regen total workload}) 

\[\mean \int\limits_0^P e^{-sV_t}\,{\rm d}t = \frac{1}{cs} \mean \sum\limits_{n=1}^{\infty} \left[e^{-scW_n}1_{\{N\geq n \}} (1- e^{-sB_n})\right]= \mean \left(\sum\limits_{n=1}^{\infty} e^{-scW_n}1_{\{N\geq n \}} \right)  \frac{1-\mean e^{-sB_1}}{cs},\]
So that
\begin{equation}
\label{regen2}
\mean \int\limits_0^P e^{-sV_t}\,{\rm d}t = \mean \left(\sum\limits_{i=1}^{N} e^{-scW_i}\right)\frac{1-\mean e^{-sB_1}}{cs}.
\end{equation}

A key remark is that the workload process is still regenerative with respect to the renewal sequence given by the epochs at which busy periods begin. Under the stability condition, the mean cycle length $\mean C$ of the workload process is finite,
 hence the stochastic mean value results still hold in this case (cf. Cohen\cite{Cohen1976}, Thm. 4.1) and we have the identities:
\[\mean e^{-sV} = \frac{1}{\mean C } \mean \int\limits_0^{C} e^{-s V_t}\,{\rm d}t, \]
and
\[\mean e^{-sW} = \frac{1}{\mean N} \mean \sum\limits_1^{N} e^{-s W_i}. \]
 We can now use these identities together with (\ref{regen2}) and $\mean C=\mean P+\mean I$, so we may write 

\[  \mean e^{-sV}= \frac{ \mean \int\limits_{0}^{P} e^{-sV_t}\,{\rm d}t+\mean I }{ \mean P+ \mean I} = \frac{\mean N \,\mean B}{\mean C} \mean e^{-sc W} \frac{1-\mean e^{-sB}}{cs\mean B} + \frac{\mean I}{\mean C} \]
Note that by definition, $P=\sum_{i=1}^N c^{-1}B_i$, with $B_i$ the i.i.d. sequence such that $B_i$ is the service requirement of the $i^{th}$ customer in a busy cycle. Hence Wald's identity gives $c\mean P=\mean N\,\mean B$, and using in addition  $\frac{\mean P}{\mean C}= \rho$, $\frac{\mean I}{\mean C}=1-\rho$ , we can rewrite the above as

\[\mean e^{-sV}= \rho \,\frac{1-\mean e^{-sB}}{s\mean B}\mean e^{-scW} + (1-\rho),\,\, \mathcal Re\,s\geq 0.\]
This can immediately be inverted to give the desired relation (\ref{queueing relation formula}).\end{proof}
\end{section}

\vspace{0.3cm}

\begin{section}{Duality between the insurance and queueing processes \label{S4}}
It is well known that there are duality relations between the classical $G/G/1$ queue and the corresponding classical Sparre Andersen insurance risk model, with independence between service requirements (respectively claim sizes) and inter-arrival times. In this case `corresponding' means: the same inter-arrival distributions, the service requirement distribution equals the claim size distribution, the service rate $c$ is the same as the premium rate.
There are two versions of the duality result (cf. Asmussen and Albrecher \cite{AsmussenRuin}, p. 45, 161):
\begin{equation}\label{waiting time duality} (i) ~~  \Psi_0(u)= \prob(cW>u), \end{equation}

\begin{equation} \label{delayed duality} (ii) ~ \Psi(u)= \prob(V>u). \end{equation}

Here $\prob(cW >u)$ is the tail of the amount of work as seen by an arriving customer in equilibrium, and $\prob(V >u)$ is the tail of the steady-state workload in the $G/G/1$ queue. $\Psi_0(u)$ is the ruin probability in the Sparre Andersen model, when at time $t=0$ the capital is $u$ and a new inter-arrival time begins, i.e., $t=0$ is an arrival epoch.
$\Psi(u)$ is the ruin probability when the risk process is started {\em in stationarity}, i.e., $t=0$ is independent of the process itself.
In this case the time elapsed until the first claim arrives has a residual distribution. We will call $\Psi_0(u)$ the ordinary ruin probability and $\Psi(u)$ the delayed ruin probability.

We pose and answer three questions in this section, for the dependencies under consideration
(between service requirement and subsequent inter-arrival time, respectively between inter-claim time and subsequent claim size):
\\
(1) Does the duality relation (\ref{waiting time duality}) still hold?
\\
(2) Does the duality relation (\ref{delayed duality}) still hold?
\\
(3) Does the relation between steady-state workload and waiting time from Theorem~\ref{queueing relation} translate
to a relation between delayed ruin probability and ordinary ruin probability, just as it does in the independent case (cf. p. 69 of Grandell
\cite{Grandell})?

The answer to question (1) is immediately seen to be positive,
as shown in Asmussen and Albrecher \cite{AsmussenRuin} p.45, because this relation uses only the random walk structure of the risk/queueing process embedded at arrival epochs, which is preserved in the model we study ($B_i$ and $A_i$ only appear in the random walk via the difference $B_i-cA_i$).
The Laplace transform of the ruin probability now immediately follows from the waiting time LST in Theorem~\ref{theorem delay},
by observing that
the relation $\Psi_0(u)=\mathbb P(cW>u)$ becomes in terms of transforms: $\Psi_0^*(s)=\frac{1}{s}\left(1-\mathbb Ee^{-scW}\right)$.
Hence we have:
\begin{Corollary} The Laplace transform of $\Psi_0(u)$, $\Psi_0^*(s):= \int_0^{\infty} e^{-su}\Psi_0(u)\,{\rm d}u$ equals
\[ \Psi_0^*(s)=\frac{1}{s}\left[1-\frac{\prod_{\tilde{s}^-_j}(1-\frac{cs}{\tilde{s}^-_j})}{\prod_{s^-_k}(1-\frac{cs}{s^-_k})}\right]. \]
\end{Corollary}
Notice that, as mentioned in the Introduction,
this result was also obtained in Constantinescu et al.~\cite{Cstescu}, using operator theory.

We shall prove that the answer to question (3) is also affirmative. In combination with the duality relation (\ref{waiting time duality}),
this implies that the answer to question (2) is also affirmative: the duality relation (\ref{delayed duality}) still holds in the dependent case.

\vspace{0.2cm}

For the purpose of studying the relation between the ordinary and the delayed ruin functions below, we assume that the pair (A,B) has a joint density, $f_{A,B}(r,z)$.\newline
Let $\phi_0(u):=1-\Psi_0(u)$ and $\phi(u):=1-\Psi(u)$ be the survival functions for the ordinary risk process and for its stationary version, respectively. In addition, denote by $u$ the initial capital, 
and let $\alpha:=\frac{1}{\mean A}$ be the arrival rate of claims.

\begin{Theorem}\label{delay relation ruin}
The relation between the survival functions for the two versions of the ruin process is
\[\phi(u)=\phi(0)+ \frac{\alpha}{c} \int\limits_{v=0}^{\infty}\int\limits_{w=0}^u\phi_{0}(u-w)\int\limits_{z=w}^{\infty} f_{A,B}(v,z)\,{\rm d}z\, {\rm d}w\, {\rm d}v. \]
\end{Theorem}

\vspace{0.2cm}

\noindent Let us make some remarks about this formula before proving it.

 \begin{remark}\label{reciprocal remark}In the stationary version of the ruin process, the first claim arrival happens after a time distributed as the residual inter-arrival time. Because of the correlation between claim sizes and their inter-arrival times, the claim size that corresponds to the residual arrival time will have a distinguished distribution; therefore let us denote the first pair by $(A^{res},B^*)$. Regarding the density function, it can be shown that (see Lemma \ref{lemma residual density})

 \begin{equation}\label{reciprocal density}f_{A^{res},B^*}(r,z)= \alpha\int\limits_{v=r}^\infty f_{A,B}(v,z)\,{\rm d}v.\end{equation}
\end{remark}

\vspace{0.2cm}
\noindent \begin{remark}\label{marginal term} The double integral that appears in the last term from Theorem \ref{delay relation ruin} above:
\[\int\limits_{v=0}^{\infty}\int\limits_{z=w}^{\infty} f_{A,B}(v,z)\,{\rm d}z\,{\rm d}v \] equals the marginal tail of a claim size, $1-F_B(w)$. If we replace this in the relation from Theorem \ref{delay relation ruin}, we obtain the same formula as in Grandell\cite{Grandell} p.69:

\begin{equation}\label{Grandell relation} \phi(u)=\phi(0)+ \frac{\alpha\mean B}{c}\int\limits_{w=0}^u\phi_{0}(u-w)\frac{1-F_B(w)}{\mean B}\,{\rm d}w. \end{equation}
This is also known as Tak\'acs' formula (see  \cite{Franken}, Corollary 4.5.4). (\ref{Grandell relation}) shows that only the marginal residual service requirement appears in this relation between
$\phi(\cdot)$ and $\phi_0(\cdot)$, even if we have the correlation between a pair $(A,B)$.

By using the fact that $\phi(u)$, $\phi_0(u)\rightarrow 1$ as $u\rightarrow \infty$, together with dominated convergence, to argue that it is allowed to interchange limit and integration, one can easily show that $\phi(0)=1-\frac{\alpha\mean B}{c}$. Now observe that Relation (\ref{Grandell relation})
between delayed and ordinary survival function
is the precise counterpart/equivalent of
relation (\ref{queueing relation formula}) between the workload and waiting time distributions.
\end{remark}

\begin{proof}[Proof of Theorem \ref{delay relation ruin}] We follow the derivation that Grandell\cite{Grandell} (p. 69, see also p.5) has given for the case when $A$ and $B$ are independent. Starting with the stationary risk process, we condition on the arrival time of the first claim, together with its size:

\[ \phi(w)=\int\limits_{r=0}^{\infty} \int\limits_{z=0}^{w+cr}\phi_0(w+cr-z)f_{A^{res},B^*}(r,z)\,{\rm d}z\,{\rm d}r.  \]
Using (\ref{reciprocal density}) we obtain: 
\[\phi(w)=\alpha \int\limits_{r=0}^{\infty} \int\limits_{v=r}^{\infty} \int\limits_{z=0}^{w+cr}\phi_0(w+cr-z)  f_{A,B}(v,z)\,{\rm d}z\,{\rm d}v\,{\rm d}r. \]
By changing the order of integration between variables $v$ and $r$,  we have:
\[\phi(w)=\alpha \int\limits_{v=0}^\infty \int\limits_{r=0}^v\int\limits_{z=0}^{w+cr}f_{A,B}(v,z) \,\phi_0(w+cr-z)\,{\rm d}z\,{\rm d}r\,{\rm d}v.\]
We use the change of variable $x:=w+cr$:
\begin{equation}\label{primitive}\phi(w)=\frac{\alpha}{c} \int\limits_{v=0}^\infty \int\limits_{x=w}^{w+cv}\int\limits_{z=0}^{x}f_{A,B}(v,z)\,
\phi_0(x-z)\,{\rm d}z\,{\rm d}x\,{\rm d}v.\end{equation}
Let us take the derivative of $\phi(w)$. In Lemma \ref{differentiation} in the Appendix we argue that this is allowed.
\[\phi'(w)=\frac{\alpha}{c} \int\limits_{v=0}^\infty \left[\int\limits_{z=0}^{w+cv}f_{A,B}(v,z)\,\phi_0(w+cv-z){\rm d}z - \int\limits_{z=0}^w f_{A,B}(v,z)\,\phi_0(w-z) {\rm d}z \right] {\rm d}v  \]
\[=\frac{\alpha}{c} \phi_0(w) - \frac{\alpha}{c} \int\limits_{v=0}^\infty \int\limits_{z=0}^{w}f_{A,B}(v,z)\,\phi_0(w-z)\,{\rm d}z\,{\rm d}v.  \]
Here we replaced the first term in the right-hand side by virtue of the renewal equation for the {\em ordinary} survival probability. We can now integrate $w$ between $0$ and $u$:

\begin{equation}\label{integral surv}\phi(u)-\phi(0)= \frac{\alpha}{c}\int\limits_{w=0}^u \phi_{0}(w)\,{\rm d}w -\frac{\alpha}{c}\int\limits_{w=0}^u \int\limits_{v=0}^{\infty}\int\limits_{z=0}^w\phi_0(w-z)\,f_{A,B}(v,z)\,{\rm d}z\,{\rm d}v\,{\rm d}w. \end{equation}

Let us focus on the last term from (\ref{integral surv}), to be called $L$.
Integration over $v$ yields, with $f_B(\cdot)$ the density of the service requirement distribution $F_B(\cdot)$:

\begin{equation}\label{intermed term}
L = \frac{\alpha}{c} \int\limits_{w=0}^u \int\limits_{z=0}^{w} \phi_0(w-z) f_{B}(z) {\rm d}z\,{\rm d}w. \end{equation}
Partial integration gives:
\begin{eqnarray}
L &=& \frac{\alpha}{c} \int\limits_{w=0}^u \phi_0(0) F_B(w)\, {\rm d}w + \frac{\alpha}{c} \int_{w=0}^u \int_{z=0}^w F_B(z) \phi'_0(w-z)\, {\rm d}z \,{\rm d}w
\nonumber
\\
&=&
\frac{\alpha}{c} \int\limits_{w=0}^u \phi_0(0) F_B(w) \,{\rm d}w +\frac{\alpha}{c}
\int_{z=0}^u F_B(z) \int_{w=z}^u \phi'_0(w-z)\, {\rm d}w\, {\rm d}z  \notag
\\
&=&
\frac{\alpha}{c} \int_{w=0}^u \phi_0(0) F_B(w)\, {\rm d}w +\frac{\alpha}{c}
\int_{z=0}^u F_B(z) [\phi_0(u-z) - \phi_0(0)] \,{\rm d}z \notag
\\
&=&
 \frac{\alpha}{c}
\int_{z=0}^u F_B(z) \phi_0(u-z)\, {\rm d}z . \label{L term}
\end{eqnarray}
Substitution of (\ref{L term}) in (\ref{integral surv}) gives (\ref{Grandell relation}) and thus the result of the proposition.
\end{proof}
\end{section}

\vspace{0.3cm}

\begin{section}{Examples and Numerical results\label{Section numerics}}

In this section we present examples of dependence structures which are tractable and have a probabilistic interpretation. We also numerically illustrate the effect of correlations on the waiting time distribution/ruin probability.
Throughout the section we take for simplicity $c=1$.

A comprehensive survey of multivariate matrix-exponential distributions (MVME) can be found in Bladt and Nielsen~\cite{Bladt:MVME}. As a special subclass of these, Kulkarni~\cite{Kulkarni:MPH*} introduced multivariate phase-type (MPH) distributions (see also Assaf et al.~\cite{Assaf:MPH}). In the bivariate case, these are defined as follows: Consider a continuous-time Markov chain $X(t)$, with finite state space $\mathcal S$, with an absorbing state $\Delta$, and generator matrix

\[ \textbf Q =\left(\begin{tabular}{cc}$Q$ & $-Q\boldsymbol{1}$ \\0 &0 \end{tabular}\right)\]

\noindent together with a reward matrix $(r^{(j)}_x)_{x,j}$, $r^{(j)}_x\geq 0$ for $x\in \mathcal S\backslash\{\Delta\} $, $j=1,2$. Assume that as long as we stay in state $x$, we earn at rate vector $\textbf{r}_x=(r_x^{(1)},r_x^{(2)})$. We look at the bivariate distribution of the random vector $(Z_1,Z_2)$, where the marginals of this vector are defined to be the total accumulated rewards until absorption:
\begin{displaymath} Z_k=\int_0^\zeta r^{(k)}_{X(t)}{\rm d}t, \end{displaymath} with $\zeta$ the time to absorption. Remark that $Z_k$ can be rewritten as

\begin{equation}\label{MPH*} Z_k=\sum_{i=1}^M r_{X_i}^{(k)}H_{i},\;\;\;k=1,2,\end{equation}
$M$ being the number of jumps until absorption of the embedded discrete-time Markov chain $X_i$ and $H_{i}$ the holding time in state $X_i$. The $H_{i}$'s are independent exponentials with rates $-Q_{X_iX_i}$. 
The dependence structure between $Z_1$ and $Z_2$ is thus given by the underlying continuous-time Markov chain $X(t)$. 
That this is indeed a subclass of MVME, follows from ~\cite{Bladt:MVME}, Theorem 4.1.

As a special case of Kulkarni's bivariate-phase type distributions, one can obtain a fairly large class of distributions by a partial decoupling of the bivariate phase-type: For the discrete-time Markov chain $X_i$, and for a fixed $i$, let $H_{i}^{(1)},H_{i}^{(2)}$ be  independent, having exponential distributions with rates $\lambda_{X_i}$ and $\mu_{X_i}$, respectively. Without loss of generality we can consider $r^{(k)}_{X_i}=1$, $k=1,2$ and set

 \[ A=\sum_{i=1}^M H_{i}^{(1)}, \;\;\; B=\sum_{i=1}^M H_{i}^{(2)}. \]

 The difference with Formula (\ref{MPH*}) is that now  the dependence structure is given only by the common underlying discrete-time Markov chain $X_i$. Furthermore, if we assume the jump rates to be the same in each state, i.e. $H_{i}^{(1)}\sim$ exp$(\lambda)$, $H_{i}^{(2)}\sim$ exp$(\mu)$, then the number of jumps $M$ before absorption is a sufficient statistic for the joint distribution of $(A,B)$. More precisely, conditional on $M$, $A$ and $B$ are independent Erlang$(M,\lambda)$, Erlang$(M,\mu)$ respectively.

 \begin{remark}
 This dependence structure can be realized as in the description of Kulkarni's class. More precisely, we obtain the partial decoupling by doubling all states of the underlying Markov Process: replace each transient state $x$ with $x_1$, $x_2$ and allow only the corresponding component of $(A,B)$ to increase while in state $x_i$ (formally, put $r^{(1)}_{x_1}=r^{(1)}_{x}$, $r^{(1)}_{x_2}=0$ and similarly $r^{(2)}_{x_1}=0$, $r^{(2)}_{x_2}=r^{(2)}_{x}$). Extend the transition matrix of the Markov Chain such that after visiting state $x_1$, it always jumps to state $x_2$ and thereafter jumps according to the original transition matrix. 
 \end{remark}
 If we denote by $\boldsymbol\alpha$ the initial distribution of $\{X_n\}$, by $T$ the transient component of its transition matrix, and by $\boldsymbol t$ the vector of exit probabilities, then by conditioning on $M$ we obtain the following result as a probabilistic alternative to Theorem 3.2 in Bladt\&Nielsen~\cite{Bladt:MVME}:
 \begin{Lemma}

 \begin{itemize}  \item[a)]  The Laplace-Stieltjes transform of $(A,B)$ is: \[\mean e^{-s_1A-s_2B}= \boldsymbol\alpha'\left[\frac{(\lambda+s_1)(\mu+s_2)}{\lambda\mu}I-T\right]^{-1}\textbf{t}.\]

\item[b)]The transform $\mean e^{-sY}$ of the difference $(B-A)$,  is a rational function of the form $\frac{f(s)}{g(s)}$, with $f$ and $g$ polynomial functions such that $\deg(f)<\deg(g)$.

\end{itemize}
\label{LS Lemma}
\end{Lemma}
\noindent
\textit{Proof:} see Appendix.

\vspace{0.2cm}

\noindent\textbf{Examples:}

$\boldsymbol{1}$. Kibble and Moran's bivariate Gamma distribution (Kotz et al.~\cite{Kotz2000}) can be realized as above. Consider the state space $\{1,...,m,\Delta\}$. Assume the Markov Chain $X_n$ starts in $1$ and jumps from $i$ to $i+1$ w.p. $p$ or stays in state $i$ w.p. $1-p$. Furthermore, assume the same rates for the holding times in every state: $H_n^{(1)}\sim$exp$(\lambda)$, $H_n^{(2)}\sim$exp$(\mu)$, for $\lambda$, $\mu>0$. 
Hence this distribution is the $m-$fold convolution of Kibble and Moran's bivariate exponential with itself (cf.~\cite{Kotz2000}), where this bivariate exponential distribution can be represented as

\[\left(Erlang(\lambda,M),Erlang(\mu,M)\right),\]
with $M$ having a geometric distribution. In the insurance risk setting, the analysis for this example has
been done in Ambagaspitiya~\cite{amba:risk} and also in Constantinescu et al.~\cite{Cstescu} using operator theory. The Laplace transform of the ordinary ruin probability $\Psi_0(u)$ is given by 
\[ \Psi_0^*(s)=\frac{1}{s}\left[1-\frac{(1-\frac{s}{b})^m}{\prod_{s_k}(1-\frac{s}{s_k})}\right],\]
with $b$ the pole of order $m$ of $1-\mean e^{-sY}$ such that $\mathcal Re\;b<0$.

$\boldsymbol{2}$. Cheriyan and Ramabhadran's bivariate Gamma is another  example of Kulkarni's bivariate phase-type. This was also analyzed in Ambagaspitiya~\cite{amba:risk}
in the insurance risk setting.

For nonnegative integers $m_0,m_1,m_2$, consider the state space $\mathcal S=\{1,...,m_0+m_1+m_2,\Delta\}$, with the set of transient states partitioned as: $\mathcal S\backslash\{\Delta\}=\mathcal S_0\cup\mathcal S_1\cup \mathcal S_2$ with $\mathcal S_0=\{1,...,m_0\}$, $\mathcal S_1=\{m_0+1,...,m_0+m_1\}$, $\mathcal S_2=\{m_0+m_1+1,...,m_0+m_1+m_2\}$. The chain starts in state 1 and jumps from state $i$ to $i+1$. The jump rates are $\beta_k$ while in state $x\in\mathcal S_k$, $k\in\{0,1,2\}$. The reward rates in state $x$ are $ r_x^{(1)}=r_x^{(2)}=1$ for $ x\in \mathcal S_0$; $r_x^{(1)}=1, r_x^{(2)}=0$ for $x\in \mathcal S_1$, and $ r_x^{(1)}=0$, $r_x^{(2)}=1$ for $ x\in \mathcal S_2$. Then the bivariate total accumulated reward has a distribution of the form

\[(A,B)\deq (Z_0+Z_1,Z_0+Z_2),\]
where $Z_k$ are mutually independent $\sim$Erlang$(m_k,\beta_k)$, $k\in\{0,1,2\}$.

$\boldsymbol{3}$. In the class of MVME, it is possible to achieve negative correlation as well. Consider $M$ to be a discrete random variable with finite support: $M\in \{ 1,...,K\}$, for $K$ some positive integer. Negative correlation can be achieved if we consider the following mixture of Erlang distributions:
\[(A,B)_-\deq \left(Erlang(\lambda,M),Erlang(\mu,K-M+1)\right).\]
For more examples of negatively correlated phase-type distributions, we refer to \cite{BoNegCor}.
\vspace{0.2cm}

\textbf{Stochastic ordering results.} We compare the tails of the waiting times for the mixed Erlang distributions in the  following scenarios: the negatively correlated one from Example 3 versus the positively correlated case

\[ (A,B)_+\deq \left(Erlang(\lambda,M),Erlang(\mu,M)\right),\]
 and the corresponding independent pair obtained by sampling twice from the distribution of $M$; i.e. for $M_1$ and $M_2$ i.i.d. copies of $M$

\[ (A,B)_0\deq \left(Erlang(\lambda,M_1),Erlang(\mu,M_2)\right).\]
Here $M$ is taken to have finite support, as in Example 3 above.


  Denote respectively by $D_-$, $D_+$ and $D_0$, the differences $A-B$ in the three scenarios above. In Theorem \ref{order thm} below  we show that under a mild assumption on the distribution of $M$, there exists \emph{convex ordering} between the random variables $D_+$, $D_0$ and $D_-$. For two r.v.'s $X$ and $Y$, $X\preceq_{cx} Y$ means, by definition, that for any arbitrary convex function $\varphi(x)$,

\begin{equation}\label{cx functional}\mean \varphi(X)\leq \mean \varphi(Y).\end{equation}
  For more about the notion of convex order and other related stochastic orderings, we refer the reader to \cite{Stoyan&Daley}, Ch. 1.
 Before we give the result, let us recall a useful criterion (cf.\cite{Stoyan&Daley}, Prop. 1.5.1):

\begin{Proposition}[Karlin \& Novikoff's cut criterion] \label{KarlinNovikoff} For $X$, $Y$ r.v.'s with c.d.f.'s $F_X$ and $F_Y$ respectively, and finite first moments, assume that $\mean X=\mean Y$, and that there exists an $x_0$ such that $F_X(x)\leq F_Y(x)$, for $x\leq x_0$ and $F_X(x)\geq F_Y(x)$ for $x\geq x_0$. Then $X\preceq_{cx} Y$.
\end{Proposition}

\begin{Theorem}\label{order thm}
\begin{equation}\label{positiveorder} D_+\preceq_{cx} D_0. \end{equation}

 Moreover, if $M$ has a symmetric distribution,  $M\deq K+1-M$, then we also have
\begin{equation}\label{negativeorder} D_0 \preceq_{cx} D_-. \end{equation}

\end{Theorem}

\begin{proof}
Let $C_i^\lambda$ and $C_j^\mu$ respectively be Erlang$(i,\lambda)$ and Erlang$(j,\mu)$ distributed  random variables independent of each other, for $i=1,..,K$; also denote $\pi_i :=\prob (M=i)$.

 We will first prove $icx$ ordering, that is the functional inequality (\ref{cx functional}) is restricted to \emph{increasing} convex functions $\varphi$. This together with the fact that the expected values of $D_-$, $D_+$ and $D_0$ are the same implies $cx$ ordering (see \cite{Stoyan&Daley}, Thm. 1.3.1, p.9).

 Take $\varphi$ to be any convex and increasing function. Firstly, we prove (\ref{positiveorder}), that is, we must show that $\mean \varphi(D_+) \leq \mean \varphi(D_0)$, or equivalently,

\[ \sum_{i=1}^K \pi_i \mean \varphi(C_i^\lambda - C_i^\mu) \leq \sum_{i=1}^K\sum_{j=1}^K \pi_i\pi_j\mean \varphi(C^\lambda_i-C^\mu_j). \]
Let us put for simplicity $\varphi(i,j):=\mean \varphi(C^\lambda_i-C^\mu_j)$, so we can rewrite the above as

\begin{equation}\label{association}
\sum_{i} \pi_i \varphi(i,i) \leq \sum_i\sum_j  \pi_i\pi_j\varphi(i,j).
\end{equation}
Note that (\ref{association}) is an association type of inequality, similar to Cebishev's inequality (see \cite{AsmussenRolski}, Lemma 2.3 and the references therein). Using that the $\pi_j's$ form a probability distribution, we can further rewrite (\ref{association})

\[ \sum_i\sum_j \pi_i\pi_j \varphi(i,i) \leq \sum_i\sum_j  \pi_i\pi_j \varphi(i,j) \]
\begin{equation}\label{tableaux1}\Leftrightarrow \sum_i\sum_{j>i} \pi_i\pi_j \left[\varphi(i,i)-\varphi(i,j)\right] \leq \sum_m\sum_{l<m} \pi_m\pi_l [\varphi(m,l) - \varphi(m,m)]. \end{equation}

Remark that there is an equal number of terms on the two sides of (\ref{tableaux1}) because we sum over indices that lie respectively above and  below the main diagonal of the tableaux $(\varphi(i,j))_{i,j}$.
We are done as soon as we show that the inequality holds for a one-to-one correspondence between these indices; more precisely, for the correspondence $(i,j)\leftrightarrow (j,i)$, $j>i$, we will prove that

\begin{equation}\label{modular1}\varphi(i,i) - \varphi(i,j) \leq \varphi(j,i) - \varphi(j,j),\end{equation} 
that is, (\ref{tableaux1}) holds term by term, and remark that the coefficients $\pi_i\pi_j$ and $\pi_j\pi_i$ cancel against each other. Put $u:=j-i$ and
denote

\[\gamma(x):= \mean \varphi(x+C^\lambda_i-C_j^\mu).\]

\noindent Obviously, $\gamma(x)$ is increasing and convex, because $\varphi$ is.
Consider the decomposition of $C_j^{\mu}$ and $C_j^{\lambda}$ as sums of independent r.v.'s $C_j^\mu:= C_i^\mu+C_u^\mu$, and $C_j^\lambda:= C_i^\lambda+C_u^\lambda$ with $C_u^\mu$, $C_u^\lambda$ Erlang distributed of order $u$ and rates $\mu$ and $\lambda$, respectively.
By conditioning on $C_u^\lambda$ and $C_u^\mu$, we can write

\[\varphi(j,i) = \mean \{\mean [\varphi(y+C_i^\lambda - C_i^\mu - x +x) | C_u^\lambda=y,C_u^\mu=x]  \} \Leftrightarrow    \]

\[\varphi(j,i) = \mean \{\mean[\gamma(y+x)|C_u^\lambda=y,C_u^\mu=x ]\} = \mean \gamma(C_u^\lambda + C_u^\mu).\]

\noindent  Similarly, we obtain $\varphi(i,i) = \mean\gamma(C_u^\mu)$ and $\varphi (j,j)=\mean \gamma(C_u^\lambda)$, so that (\ref{modular1}) becomes


\begin{equation}\label{modular intermed} \mean \gamma(C_u^\mu)+ \mean \gamma(C_u^\lambda) \leq \mean \gamma(C_u^\lambda + C_u^\mu) + \gamma(0) .\end{equation}

All boils down to proving (\ref{modular intermed}). In order to achieve this,
let $X$ be a r.v. with a Bernoulli(1/2) distribution and let $c_\mu\neq c_\lambda$ be two arbitrary positive constants. Consider the following r.v.'s

\[ Z_1:= c_\lambda X +  c_\mu X, \,\,\,\, Z_2:=c_\lambda X + c_\mu(1-X). \]
We have the following identities in distribution

\[Z_1\deq \frac{1}{2}[\delta_0 + \delta_{c_\lambda+c_\mu}] ,\;\;\; Z_2\deq \frac{1}{2} [ \delta_{c_\lambda} +  \delta_{c_\mu}], \]
with $\delta_x$ being the Dirac measure at $x$. Now it follows easily from the cut criterion in Proposition \ref{KarlinNovikoff} above that $Z_2\preceq_{cx} Z_1$. Hence, in particular, we can choose $\gamma(x)$ as a test function to obtain

\[ \mean \gamma(c_\lambda X + c_\mu(1-X) ) \leq \mean \gamma(c_\lambda X +  c_\mu X). \]
Because $X$ is a Bernoulli(1/2), the inequality above becomes
\[  \gamma(c_\lambda) + \gamma(c_\mu)\leq  \gamma(c_\lambda+c_\mu) +\gamma(0).   \]
Finally, taking the double mixture over $c_\lambda$ and $c_\mu$ according to the distributions of $C_u^\lambda$ and $C_u^\mu$ respectively, shows that (\ref{modular intermed}) is true, and this proves (\ref{positiveorder}).





\vspace{0.2cm}
Now, for inequality (\ref{negativeorder}) we have to prove that $\mean \varphi(D_0)\leq \mean \varphi(D_-)$, that is, keeping the same notation as in (\ref{association}),

\[ \sum_i\sum_j \pi_i\pi_j \varphi(i,j) \leq \sum_i\sum_j \pi_i\pi_{j} \varphi(i,K+1-i), \]
and  upon regrouping terms it becomes

\[ \sum_i\sum_{j:\,j<K+1-i} \pi_i\pi_j [\varphi(i,j) - \varphi(i,K+1-i)] \leq \sum_m\sum_{l:\,l>K+1-m} \pi_m\pi_l [\varphi(m,K+1-m) - \varphi(m,l)].    \]
This is the analogue of (\ref{tableaux1}). Again, it suffices to prove the term by term inequalities similar to (\ref{modular1}).
The symmetry axis in this case is the second diagonal of the tableaux. This means that the correspondence is $(i,j)\leftrightarrow (K+1-j,K+1-i)$, so
the analogue of (\ref{modular1}) that we prove is, for $i,j$ fixed, $j<K+1-i$

\begin{equation}\label{modular2} \varphi(i,j) - \varphi(i,K+1-i) \leq \varphi(K+1-j,j) - \varphi(K+1-j,K+1-i). \end{equation} 
In (\ref{modular2}) we dropped the coefficients $\pi_i\pi_j$ and $\pi_{K+1-i}\pi_{K+1-j}$ because these are equal since $M$ is assumed to have a symmetric distribution.
If we set $u=(K+1-i)-j=(K+1-j)-i$, from this point on the analysis is essentially the same. Consider the analogue of $\gamma$,
\[\eta(x):= \mean \varphi\left(x+C^\lambda_i-C_{K+1-i}^\mu\right),\]

then $(\ref{modular2})$ becomes
\[ \mean\eta(C_u^\mu) - \eta(0) \leq \mean\eta(C_u^\lambda + C_u^\mu) - \mean\eta(C_u^\lambda).  \]

This is precisely (\ref{modular intermed}) with $\gamma(x)$ replaced by $\eta(x)$, and since $\varphi$ was taken to be an arbitrary increasing convex function, the proof is complete.\end{proof}

\vspace{0.2cm}
\begin{remark} The requirement for $M$ to have a symmetric distribution may be too strong in general.
Some assumption on the distribution of $M$ is necessary but only for the ordering $D_0\preceq_{cx} D_-$.  For example, if we let $K=2$ and $M\deq\delta_1$ (Dirac mass in 1) then $D_0$ is the difference of two independent Erlang-1, whereas $D_-$ is an Erlang-1 minus an Erlang-2 so $D_-$ is $cx$-dominated in this case.

\vspace{0.2cm}
The above proof of the inequality between $D_+$ and $D_0$ does not require the finiteness of the support of $M$; $M$ discrete phase type is also a possible case in which the sums that appear in the proof become series. There are no convergence problems and we are allowed to change summation order as well, due to probabilistic interpretations. Of course there are restrictions if we look for negative correlation when $M$ has infinite support. More about this possibility can be found in Bladt and Nielsen~\cite{BoNegCor} on negatively correlated exponentials.
\end{remark}

\begin{Proposition} Let $W_-$, $W_0$, and $W_+$, be the steady-state waiting times, that correspond to the increments of the random walk distributed as $-D_-$, $-D_0$, and $-D_+$, respectively. Then we have convex ordering between the waiting times in the three scenarios
\[  W_+ \preceq_{cx} W_0\preceq_{cx} W_-.\]
\end{Proposition}

\begin{proof}
From the definition of convex ordering, $D_+\preceq_{cx} D_0$ is the same as $-D_+\preceq_{cx} - D_0$, and similarly $D_0\preceq_{cx} D_-$ is the same as $-D_0\preceq_{cx} - D_-$.
Therefore the external monotonicity result from Daley and Stoyan \cite{Stoyan&Daley} (Thm. 5.2.1, p.80) implies that the steady state workloads are convex ordered in the three scenarios, according to the increments of the random walk.
This can also be seen in the numerical tables and the plots below.\end{proof}

In Table \ref{tabletwo} below, we keep $\rho$ fixed, say $\rho=.5$, and we vary $K$.
In Table \ref{tableone} we vary the load coefficient $\rho$ and we keep the mixing distribution $M$ uniform on $\{1,...,5\}$ (i.e., $K=5$).  The tables contain the mean waiting times, their atoms at zero and $q$, the 95$\%$ quantile of the survival function/waiting time (i.e., $q$ is the value of the initial capital for which $\prob(W\leq q)=\phi_0(q)=.95$).
The plots of the tails of the ruin functions are in Figure \ref{figure 3} and Figure \ref{figure 2} below.

\begin{table}[!htbp]
\begin{center}
\begin{tabular}{c ccc ccc ccc}
\hline\hline
 $K$  &$\mean W_+$ &$\mean W_0$ &$\mean W_-$  &$\prob(W_+=0)$ &$\prob(W_0=0)$ &$\prob(W_-=0)$ & $q_+$ & $q_0$ & $q_-$ \\ \hline
 2  &0.86 &1.11 &1.36  &0.57 &0.54 &0.51 & 4.36 & 5.31 & 6.25 \\
 4 &0.68 &1.37 &2.11  &0.67  &0.58 &0.52 & 3.93 & 6.78 & 9.48 \\
 7 &0.51  &1.78 &3.22  &0.75 &0.61 &0.53 & 3.39 & 9.09 & 14.35 \\
 14 &0.31 &2.79 &5.82  &0.85 &0.64 &0.540 & 2.33 & 14.58 & 25.74 \\
\hline\hline
\end{tabular}
\end{center}
\caption{Mean waiting times, atoms at 0 and 95$\%$ percentiles for $\rho=.5$ and various values of $K$}
\label{tabletwo}
\end{table}

\begin{figure}[!htbp]
\subfigure[$\rho=.5$, $K=2$]{
\includegraphics[scale=.40]{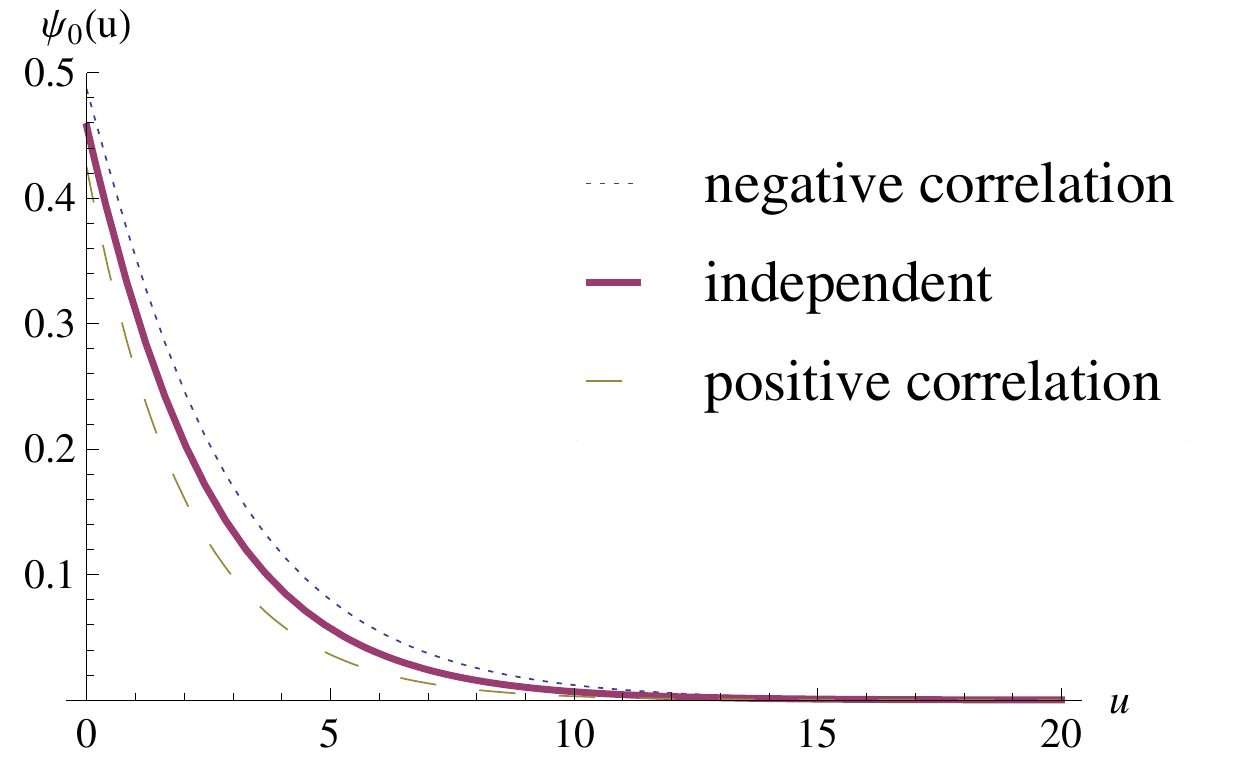}
}
\subfigure[$\rho=.5$, $K=4$]{
\includegraphics[scale=.55]{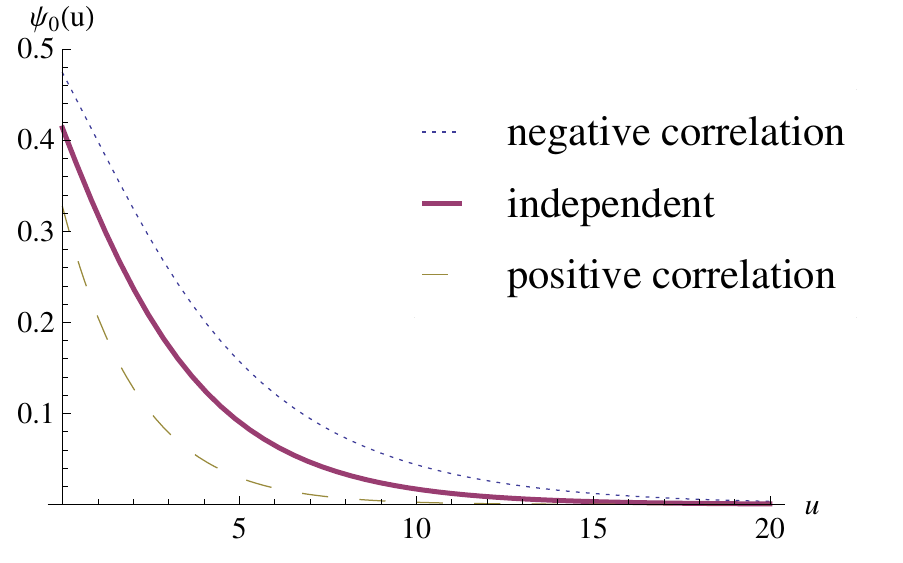}
}
\subfigure[$\rho=.5$, $K=7$]{
\includegraphics[scale=.55]{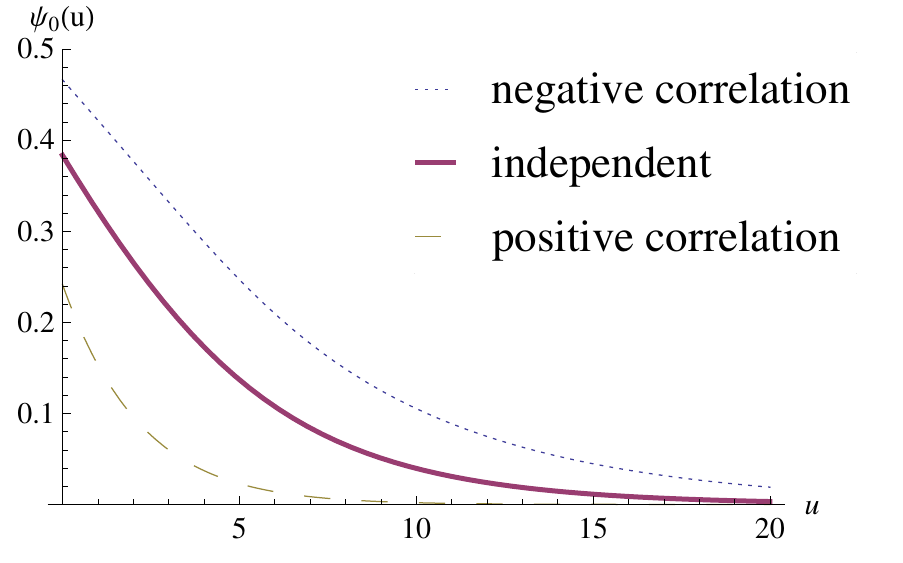}
}
\subfigure[$\rho=.5$, $K=14$]{
\includegraphics[scale=.55]{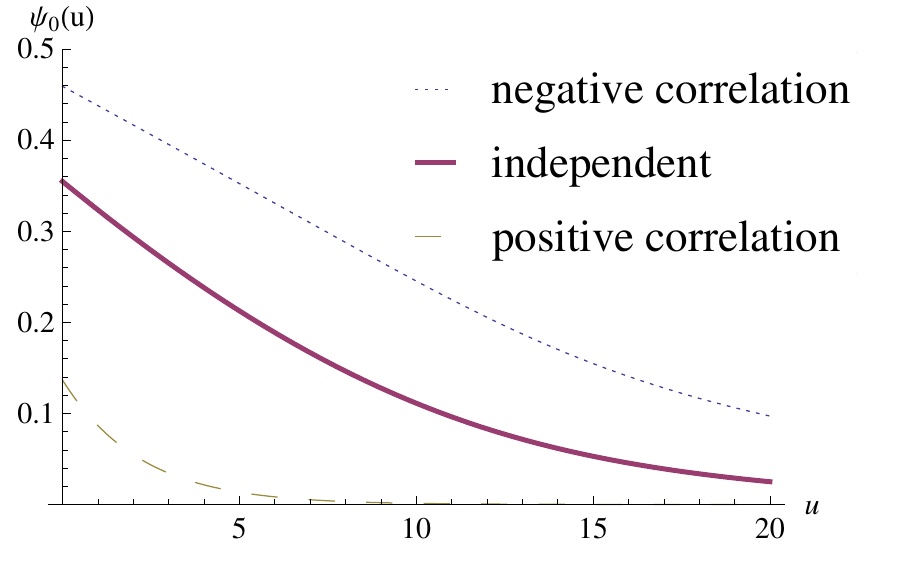}
}
\caption{$\prob(W>u)=\Psi_0(u)$.}
\label{figure 3}
\end{figure}

\begin{table}[!htbp]
\begin{center}
\begin{tabular}{c ccc ccc ccc}
\hline\hline
 $\rho$  &$\mean W_+$ &$\mean W_0$ &$\mean W_-$  &$\prob(W_+=0)$ &$\prob(W_0=0)$ &$\prob(W_-=0)$ & $q_+$ & $q_0$ & $q_-$ \\ \hline
.05 &0.01  &0.07 &0.15 &0.988  &0.96  &0.95  & 0 & 0 & 0 \\
.25 &0.12 &0.47 &0.88  &0.90  &0.82 &0.76 & 0.85 & 3.54 & 5.72\\
.5 &0.62  &1.50 &2.48 &0.70 &0.59 &0.52 & 3.74 & 7.54 & 11.1 \\
.75 &2.48 &4.77 &7.15  &0.39 &0.32 &0.27 & 10.14 & 17.81 & 25.5  \\
.95 &18.4 &31.4 &44.48 &0.08  &0.066   &0.056 & 58.26 & 97.89 & 137.58 \\

\hline\hline
\end{tabular}
\end{center}
\caption{Mean waiting times, atoms at 0 and 95$\%$ percentiles for $K=5$ and various values of $\rho$ }
\label{tableone}
\end{table}

\begin{figure}[!htbp]
\subfigure[$\rho=.05$, $K=5$]{
\includegraphics[scale=.55]{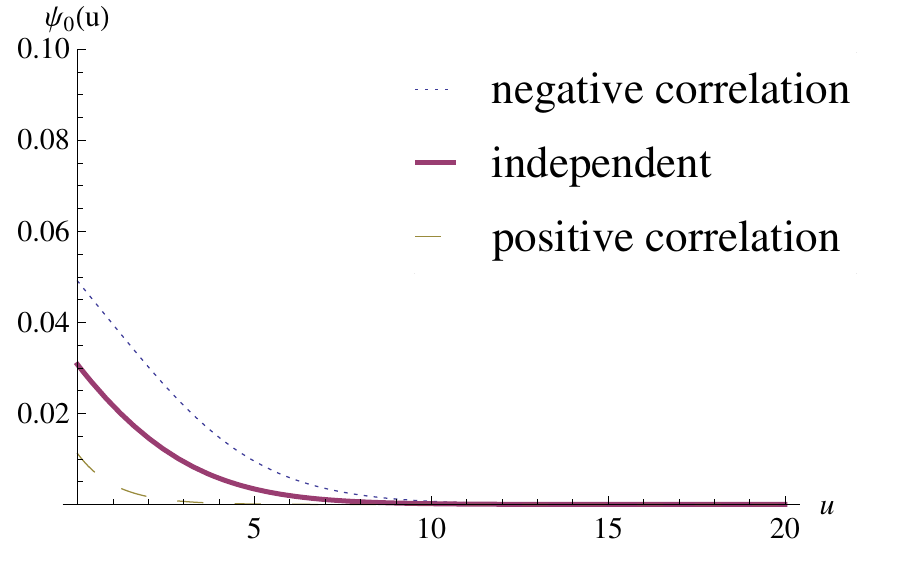}
}
\subfigure[$\rho=.25$, $K=5$]{
\includegraphics[scale=.55]{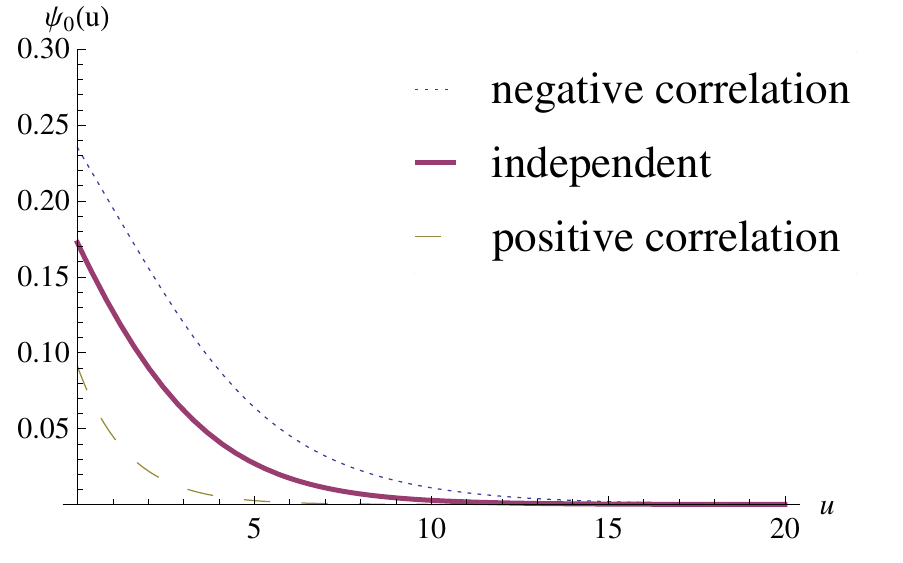}
}
\subfigure[$\rho=.5$, $K=5$]{
\includegraphics[scale=.55]{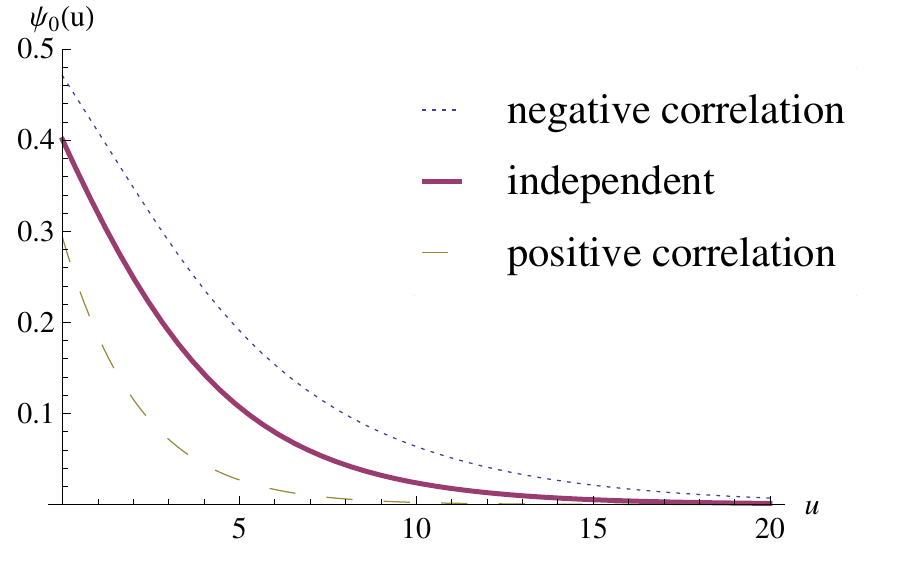}
}
\subfigure[{$\rho=.75$, $K=5$}]{
\includegraphics[scale=.55]{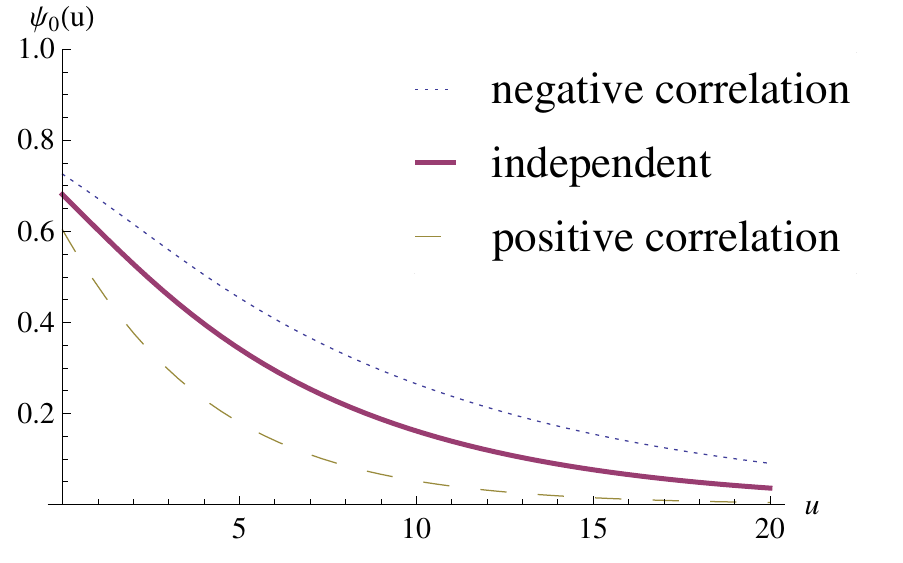}
}
\subfigure[{$\rho=.95$, $K=5$}]{
\includegraphics[scale=.40]{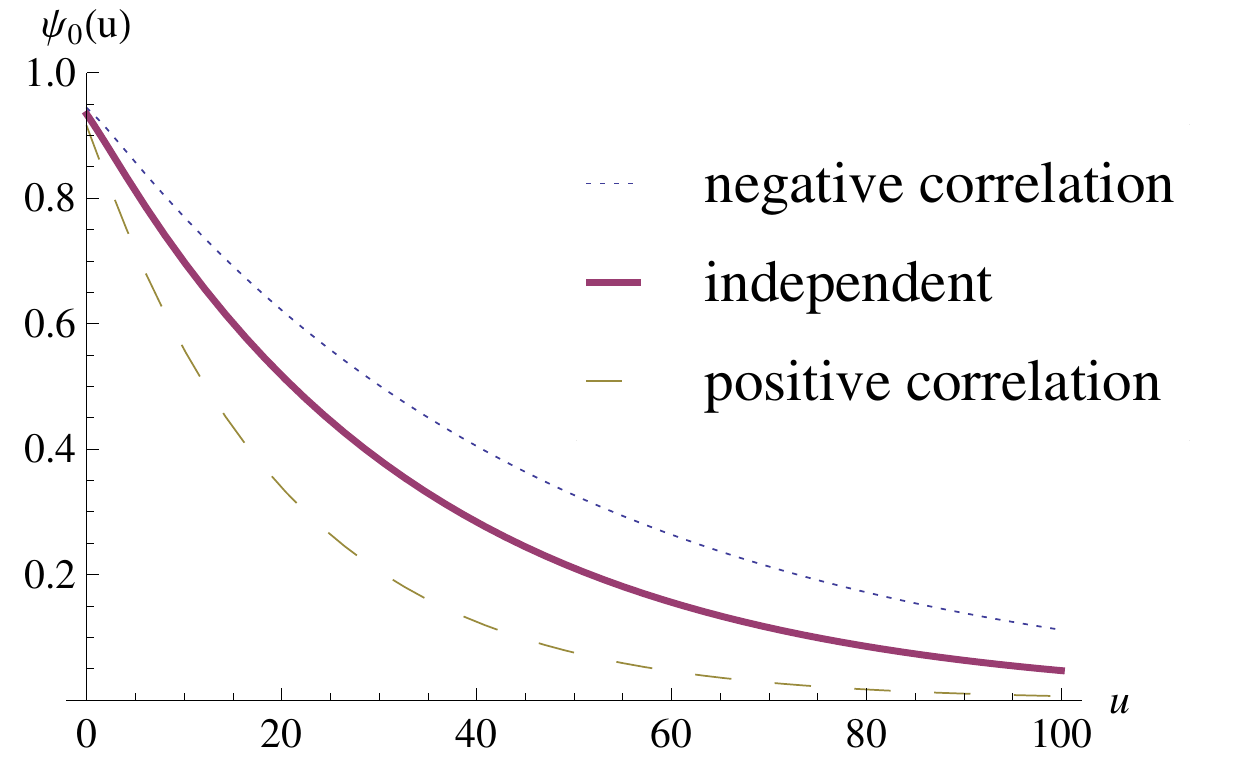}
}
\caption{$\prob(W>u)=\Psi_0(u)$.}
\label{figure 2}
\end{figure}

\end{section}

 \renewcommand{\thesection}{A}
  \setcounter{section}{1}
  \section*{APPENDIX}
\begin{ATheorem}[$Rouch\acute{e}$,~\cite{Titchmarsh}, p.116]\label{A1}If two functions $g(s)$ and $f(s)$ are analytic inside and on a closed contour $C$, and $|g(s)|>|f(s)|$ on $C$, then $g(s)$ and $g(s)-f(s)$ have the same number of zeros inside $C$. \end{ATheorem}

\begin{ATheorem}[$Liouville$,~\cite{Titchmarsh}, p.85]\label{A2} If $f(s)$ is analytic for all finite values of $s$, and as $|s|\rightarrow \infty$,

 \[f(s)=O(|s|^m),\] then $f(s)$ is a polynomial of order $\leq m$.
\end{ATheorem}
We can now formulate and prove the following lemma.
\begin{ALemma}\label{A3} Let $f(s)$ and $g(s)$ be the numerator and the denominator of $\mathbb E e^{-s(c^{-1}B-A)}$. Then $g(s)-f(s)$ and $g(s)$  have the same number of zeros in $\mathcal Re\;s\geq 0$. 
\end{ALemma}
\begin{proof} Via Rouch$\acute{e}$'s theorem, we first prove that $|g(s)|>|f(s)|$ on a suitably chosen contour in the complex plane. The fact that $f(0)=g(0)$ and that the transform is rational (so it is also analytic on a strip in $\mathcal Re$ $s<0$) suggests that we consider the following contour made up from the extended semi-circle \[\mathcal C_\epsilon:=\left\{R(\cos\varphi+i\sin\varphi);\;\varphi\in \left[-\pi/2-\arccos \epsilon,\pi/2+\arccos\epsilon\right] \right\},\] together with the vertical line segment $S:=\left\{-\epsilon+i\omega;\;|\omega|\in\left[0,R\sqrt{1-\epsilon^2} \right]\right\}.$

We show that $|g(s)|>|f(s)|$ on this contour, for $\epsilon$ sufficiently small.\newline First on $\mathcal C_\epsilon$: $\left|\frac{f(Re^{i\varphi})}{g(Re^{i\varphi})}\right|\leq \mathbb E e^{- R\cos\varphi (c^{-1}B-A) }\rightarrow \mathbb P(c^{-1}B-A=0)$ as $R\rightarrow \infty$. 
 We can assume $\mathbb P(A=c^{-1}B)<1$, else there is nothing to prove. This means 
 $|\frac{f(Re^{i\varphi})}{g(Re^{i\varphi})}|<1$ for $R$ sufficiently large.

 In order to prove the inequality on the line segment $S$, we use the stability condition: $\mathbb E(A-c^{-1}B)=\frac{d}{ds}\frac{f(s)}{g(s)}|_{s=0}>0$. So for $\epsilon$ sufficiently small, $\frac{f(-\epsilon)}{g(-\epsilon)}<\frac{f(0)}{g(0)}=1$. Then on $S$ we have: \[\left|\frac{f(-\epsilon+i\omega)}{g(-\epsilon+i\omega)}\right|=|\mathbb E e^{-(-\epsilon+i\omega)(c^{-1} B-A)}|\leq\mathbb Ee^{\epsilon(c^{-1}B-A)}|e^{-i\omega(c^{-1}B-A)}|=\frac{f(-\epsilon)}{g(-\epsilon)}<1. \]

Hence $|f(s)|<|g(s)|$ on the whole contour. These being polynomials, Rouch$\acute{e}$'s theorem \ref{A1} ensures that $g$ and $g-f$ have the same number of zeros inside $\mathcal C_\epsilon$, and since $\epsilon$ was arbitrarily small, this also holds on $\displaystyle\cap_{\epsilon>0}\mathcal C_\epsilon^\circ=\left\{s;\;\mathcal Re\;s\geq 0\right\}\cap\left\{s;\;|s|\leq R\right\}$, where $\mathcal C_\epsilon^\circ$ is the interior of $C_\epsilon$. Finally, letting $R\rightarrow\infty$, proves the assertion. \end{proof}

\begin{proof}[\textbf{Proof of Lemma \ref{LS Lemma}}]

  a) We can write the joint Laplace-Stieltjes transform by conditioning on $M$:

 \[  \mathbb E e^{-s_1A-s_2B}= \sum_{n=1} ^\infty \mathbb P(M=n) \left(\frac{\lambda}{\lambda+s_1}\right)^n \left(\frac{\mu}{\mu+s_2}\right)^n.  \]
 If we set $z=\frac{\lambda}{\lambda +s_1}\frac{\mu}{\mu+s_2}$, we can recognize the probability generating function of $M$ at $z$, call it $P_M(z)$.

$M$ has a discrete phase-type distribution with representation $(\boldsymbol\alpha, T)$ (Neuts~\cite{Neuts_mat_analyt}), such that $I-T$ is non-singular
(here $I$ is the identity matrix), and the probability vector $\boldsymbol\alpha$ is supported on the transient states. Thus \[\mathbb P(M=n)=\boldsymbol\alpha' T^{n-1}\boldsymbol t \] for $n\geq 1$, $\boldsymbol t=(I-T)\boldsymbol 1$, $P(M=0)=0$.
  If we now focus on this generating function, we have the following (Asmussen \cite{Asmussen_queues}~Prop. 4.1, p.83):

 \[P_M(z)= \boldsymbol\alpha'(\frac{1}{z}I-T)^{-1}\boldsymbol t,  \] 
 and we have proved part $a).$

 b) To see why $\mean e^{-sY}=P_M(\frac{\lambda}{\lambda-s}\frac{\mu}{\mu+s})$ is a rational function, rewrite the inverse : \[(\frac{1}{z}I-T)^{-1}=\frac{ 1}{\det(\frac{1}{z}I-T)}(\frac{1}{z}I-T)^*.\]
 Remark that the denominator $\det(\frac{1}{z}I-T)$ is a polynomial of order $|\mathcal S|-1$ (the number of transient states) in $\frac{1}{z}$, because $z^{-1}$ appears only on the diagonal of the matrix $(\frac{1}{z}I-T)$.
  $(\frac{1}{z}I-T)^*$ is the algebraic complement of $(\frac{1}{z}I-T)$ (also known as matrix of cofactors). Its entries are of the form $(-1)^{i+j}\det(M_{ij})$, where $M_{ij}$ is the matrix obtained by deleting row $i$ and column $j$ of $(\frac{1}{z}I-T)$. These are polynomials in $z^{-1}$ of order $<|\mathcal S|-1$ (because of the deleted rows and columns in the entries, the degree of the determinants of these sub-blocks as polynomials in $z^{-1}$ is always smaller than the dimension of the matrix $T$) and hence so is the bilinear form $\boldsymbol\alpha' (\frac{1}{z}I-T)^*\boldsymbol t $, which is the numerator of $P_M(z)$.
\end{proof}

\begin{ALemma}\label{differentiation} $\phi(w)$ in (\ref{primitive}) is differentiable.

\end{ALemma}
\begin{proof}
Let $h_w(v):= \int\limits_{x=w}^{w+cv}\int\limits_{z=0}^{x}f_{A,B}(v,z)\, \phi_0(x-z)\,{\rm d}z\,{\rm d}x.$ Using the triangle inequality, we have the following upper bound
\[ |h_{w+\epsilon}(v)-h_{w}(v)| \leq \int\limits_{x=w}^{w+\epsilon}\int\limits_{z=0}^x f_{A,B}(v,z) \phi_0(x-z)\,{\rm d}z\,{\rm d}x + \int\limits_{x=w+cv}^{w+cv+\epsilon}\int\limits_{z=0}^x f_{A,B}(v,z)\phi_0(x-z)\,{\rm d}z\,{\rm d}x. \]
Let us denote by $I$ and $II$ the first and the second term that appear above, respectively. If we use the fact that $\phi_0(x)\leq 1$, we find the upper bounds on $I$ and $II$:
\[I\leq \int\limits_{x=w}^{w+\epsilon}\int\limits_{z=0}^{w+\epsilon} f_{A,B}(v,z)\,{\rm d}z\,{\rm d}x = \epsilon\int\limits_{z=0}^{w+\epsilon} f_{A,B}(v,z)\,{\rm d}z,\]
and similarly,
\[ II\leq \epsilon\int\limits_{z=0}^{w+cv+\epsilon} f_{A,B}(v,z)\,{\rm d}z .\]
So if we denote $D_\epsilon(v):=\frac{h_{w+\epsilon}(v)-h_w(v)}{\epsilon}$,
\[|D_\epsilon(v)| \leq \int\limits_{z=0}^{w+\epsilon} f_{A,B}(v,z)\,{\rm d}z + \int\limits_{z=0}^{w+cv+\epsilon} f_{A,B}(v,z)\,{\rm d}z \leq 2 f_A(v), \]
and clearly the upper bound is integrable as a function of $v$. By virtue of dominated convergence
\[\phi'(w)=\lim_{\epsilon\rightarrow 0}\int\limits_{v=0}^\infty D_\epsilon(v)\,{\rm d}v = \int\limits_{v=0}^\infty \lim_{\epsilon \rightarrow 0} D_\epsilon(v)\,{\rm d}v=\int\limits_{v=0}^\infty \frac{\partial}{\partial w}h_w(v)\,{\rm d}v.   \]
\end{proof}

\begin{ALemma}\label{lemma residual density}
Under the conditions from Remark \ref{reciprocal remark}, the density of the pair $(A^{res},B^*)$ is
\[f_{(A^{res},B^*)}(r,z)= \alpha \int\limits_{v=r}^\infty f_{(A,B)}(v,z)\,{\rm d}v.\]
\end{ALemma}
\begin{proof} 

Consider the augmented pair $(\tilde A,B^*)$ which by definition has density
\[f_{(\tilde A,B^*)}(v,z):= \alpha v f_{A,B}(v,z),\]
 where $\alpha$ acts as the normalizing factor: $\frac{1}{\alpha} =\mean A = \int_z \int_v v f_{A,B}(v,z)\,{\rm d}v\,{\rm d}z.$ Let $U$ be a standard uniform r.v., independent of both $A$ and $B$. Then  $(A^{res},B^*)\deq((1-U)\tilde A,B^*)$, therefore  
conditional on $\tilde A$, $A^{res}$ is uniformly distributed over the interval $[0, \tilde A]$, so we may write in terms of density functions

\[  f_{(A^{res},B^*)}(v,z) = \int\limits_{r=v}^{\infty} \frac{1}{r}f_{(\tilde A,B^*)}(r,z) \,{\rm d}r = \alpha \int\limits_{r=v}^{\infty} f_{(A,B)}(r,z) \,{\rm d}r. \]

\end{proof}
\vspace{0.3cm}

\noindent
{\bf Acknowledgment}. The authors are indebted to Hansjoerg Albrecher, S{\o}ren Asmussen, Zinoviy Landsman, and Tomasz Rolski
for valuable discussions and useful references.

\bibliographystyle{plain}

\end{document}